\documentclass[]{interact}

\usepackage{epstopdf}
\usepackage[caption=false]{subfig}

\usepackage{color}

\usepackage[numbers,sort&compress]{natbib}
\bibpunct[, ]{[}{]}{,}{n}{,}{,}
\renewcommand\bibfont{\fontsize{10}{12}\selectfont}

\theoremstyle{plain}
\newtheorem{theorem}{Theorem}[section]
\newtheorem{lemma}[theorem]{Lemma}
\newtheorem{corollary}[theorem]{Corollary}
\newtheorem{proposition}[theorem]{Proposition}

\theoremstyle{definition}

\theoremstyle{remark}

      \def\dR{{\mathbb R}}

   \def\dZ{{\mathbb Z}}

   \def\cB{{\mathcal B}}   
      \def\cF{{\mathcal F}}
      
      \def\cL{{\mathcal L}}

\newcommand{\hg}[2]{\,\mbox{}_{#1}F_{ #2}\!\!}

\newcommand{\argu}[3]{\left[\begin{array}{c} #1\\#2\end{array} ; #3\right]}

\begin{document}

\articletype{Article}

\title{The single-indexed exceptional Krawtchouk polynomials}

\author{
\name{Hiroshi Miki\textsuperscript{a} and Satoshi Tsujimoto\textsuperscript{b}\thanks{CONTACT S. Tsujimoto Email: tsujimoto.satoshi.5s@kyoto-u.jp} and Luc Vinet\textsuperscript{c,d}}
\affil{\textsuperscript{a} Meteorological College, Asahi-Cho, Kashiwa 277 0852,
Japan;\\
\textsuperscript{b} Department of Applied Mathematics and Physics, Graduate School of Informatics,
Kyoto University, Sakyo-Ku, Kyoto 606 8501, Japan;\\
\textsuperscript{c} Centre de recherches math\'{e}matiques, Universit\'{e} de Montr\'{e}al,
PO Box 6128, Centre-ville Station, Montr\'{e}al (Qu\'{e}bec), H3C 3J7,
Canada;\\
\textsuperscript{d} Institut de valorisation des donn\'{e}es (IVADO), Montr\`{e}al (Qu\'{e}bec), H2S 3H1, Canada}
}

\maketitle

\begin{abstract}
The Darboux transformations of Krawtchouk polynomials are investigated and all possible exceptional Krawtchouk polynomials obtainable from a single-step Darboux transformation are considered.
The properties of these exceptional Krawtchouk polynomials including the Diophantine ones and the recurrence relations are obtained.
\end{abstract}

\section{Introduction}

Classical orthogonal polynomials (COPs), defined as polynomial eigenfunctions of a second order differential/($q$-)difference operator, appear and play important roles in many problems \cite{andrews1985classical}.
Generalizations of the COPs have been proposed from many viewpoints and an important one is that of the exceptional orthogonal polynomials (XOPs) .
The definition of the XOPs is almost the same as that of the COPs.
The difference lies in the absence of several degrees in XOP families although these ensembles nevertheless provide complete bases for the corresponding Hilbert spaces \cite{gomez2009extended,gomez2010extension}.  
The XOPs can be obtained from the COPs by applying iterated Darboux transformations \cite{gomez2010exceptional}.
Due to this correspondence, the Askey-scheme of the COPs has been extended to XOPs and the exceptional Askey-Wilson polynomials have been extensively studied \cite{odake2013multi,duran2015constructing,duran2020exceptional}. \par 
We shall focus on the Krawtchouk polynomials which are COPs of a discrete variable and shall consider their exceptional counterparts. The ordinary Krawtchouk polynomials have applications in many areas including signal processing, coding theory and so on \cite{levenshtein1995krawtchouk,yap2003image}.
The authors recently found that the exceptional Krawtchouk ($X$-Krawtchouk) polynomials lead to interesting continuous-time classical and quantum walks \cite{mtv}. With further applications in mind, this study motivates the examination of the properties of these $X$-Krawtchouk polynomials. 
It was already pointed out that the $X$-Krawtchouk polynomials can formally be obtained from the exceptional Meixner ($X$-Meixner) polynomials by choosing a different parametrization \cite{duran2014exceptionalc}.
In that sense, the properties of the $X$-Krawtchouk polynomials are expected to be obtainable from those of the $X$-Meixner polynomials.
However, the $X$-Krawtchouk polynomials are finite orthogonal polynomials and hence exhibit specific features that can not be seen in the $X$-Meixner case. \par 
This paper aims to characterize explicitly the $X$-Krawtchouk polynomials by applying a Darboux transformation directly to the ordinary Krawtchouk polynomials.
It should be noted that multi-step Darboux transformations are usually considered in the construction of XOPs.
However, we shall only examine here the single-step transformation so as to work out all details.
We shall thus proceed to determine the properties of these $X$-Krawtchouk polynomials including the Diophantine ones, the recurrence formulas and so on. \par 
This paper is organized as follows.
In Section \ref{sec2}, the basics of the Krawtchouk polynomials are introduced as polynomial eigenfunctions of the Krawtchouk operator and their properties are reviewed from that angle. All eigenfunctions of the Krawtchouk operator will then be considered.
In Section \ref{sec3}, we introduce the Darboux transformation and the $X$-Krawtchouk operator from the eigenfunctions of the Krawtchouk operator.
We then construct the $X$-Krawtchouk polynomials by this Darboux transformation in Section \ref{sec4}.
The properties of the $X$-Krawtchouk polynomials including their orthogonality, recurrence relations, etc. are spelled out in Section \ref{sec5}.
We conclude with a summary.

\section{Krawtchouk polynomials and their difference operator}\label{sec2}
We first review the definition and basic properties of the Krawtchouk polynomials, which are necessary for the discussion that follows.

Let $N$ be a positive integer and $p \in (0,1)$.
For $n = 0,1,2,\ldots$, the Krawtchouk polynomials of $n$th degree in monic form are given by
\begin{align}
 K_n(x) &= K_n(x;p,N)
=(-N)_n\, p^n \hg{2}{1}\argu{\!\!-x,-n}{-N}{\dfrac{1}{p}}\nonumber\\
&=\sum_{j=0}^{n}\dfrac{(-n)_{j}(-N+j)_{n-j}}{j!} \,p^{n-j}\,(-x)_j
\end{align}

with $(a)_n$ the standard Pochhammer symbol defined by
\begin{equation}
(a)_n=\begin{cases}
1& (n=0)\\
a(a+1)\cdots (a+n-1) & (n=1,2,\ldots )
\end{cases}.
\end{equation}
The monic polynomial sequence $\{K_n(x)\}_{n \in \dZ_{\ge 0}}$ satisfies the three-term recurrence relation:
\begin{align}
\begin{split}
 x K_n(x) &= K_{n+1}(x)+\left(p (N-n)+n(1-p)\right)K_n(x) \\
 &+ (N+1-n) n p(1-p) K_{n-1}(x)
\label{K:TRR}
\end{split}
\end{align}
with the initial values $K_{0}(x)=1$ and $K_1(x)=x-Np$.
In the following, we will omit the dependence of the function on $p$ and $N$ and write $F(x)=F(x;p,N)$ so long as there is no confusion.

The set consisting of $K_0(x),K_1(x),\ldots,K_N(x)$ is known to verify the discrete orthogonality relation associated to a binomial distribution: 
for $n,m \in \{0,1,2,\ldots,N\}$, 
\begin{equation}
 \sum_{x =0}^{N}  w(x) K_n(x) K_m(x) = h_n \delta_{n,m},
\end{equation}
where 
\begin{align}
\begin{split}
  &w(x) = \dfrac{N! }{x!(N-x)!} \, p^x \,(1-p)^{N-x} , \\
  &h_n = (-1)^n (-N)_n \,n! \,p^n \,(1-p)^n.
  \end{split}
\end{align}
Note that the Krawtchouk polynomials possess several symmetries some of which are expressed in the relations recorded below for later usage:
\begin{align}
\begin{split}\label{K:dual}
&K_n(x;1-p,a) = (-1)^n K_n(a - x;p,a)\quad (a\in \mathbb{R}) \\
& (-N)_x p^xK_{n}(x;p,N) = {(-N)_n p^n} K_x(n;p,N) \quad (x \in \{0,1,2,\ldots,N\}),\\
& \frac{K_{N-n}(x;p,N)}{K_{n}(N-x;p,N)} =  \dfrac{(N-n)!}{n!}(p-1)^{x-n}(-1)^N p^{N-x-n}  \quad (x \in \{0,1,2,\ldots,N\}),\\
& \left( \frac{p-1}{p}\right)^x K_{n}(x;1-p,N)=\left( \frac{1-p}{p}\right)^n \frac{ (-N)_n}{p^N(-N)_{N-n}}K_{N-n}(x;p,N).
\end{split}
\end{align}
Furthermore, the following additional formulas will also prove useful:
\begin{align}\label{K:rec_variant}
\begin{split}
(x-N) K_n(x+1;p,N) &= K_{n+1}(x;p,N) + (2n-N)(1-p) K_{n}(x;p,N) \\
& + n(n-N-1)(1-p)^2 K_{n-1}(x;p,N),\\
K_n(x+1;p,N) & = K_n(x;p,N)+ n K_{n-1}(x;p,N-1),\\
K_n(x+1;p,N+1)&=K_n(x;p,N) + n (1-p) K_{n-1}(x;p,N),\\
K_n(x;p,N+1)&=K_n(x;p,N) + n (-p) K_{n-1}(x;p,N).
\end{split}
\end{align}
There exists a factorization property of the Krawtchouk polynomials when $n\ge N+1$ as
\begin{align}
 K_n(x) = K_{N+1}(x) \,Q_{n-N-1}(x),
\label{K:factorprop}
\end{align}
where 
\begin{align}
\begin{split}
 &K_{N+1}(x)=K_{N+1}(x;p,N) = x(x-1)(x-2)\cdots(x-N),\\
 &Q_{m}(x) = K_{m}(x-N-1;p,-N-2).
 \end{split}
\end{align}
It should be mentioned that the relation \eqref{K:factorprop} is usually called ``Diophantine property'' \cite{bruschi2009additional,chen2010hypergeometric} of the Krawtchouk polynomials.

Let $\cL$ be the second order difference operator defined by
\begin{align}
\label{op:Krawtchouk}
 \cL = p(N-x) (T-I) + x(1-p)(T^{-1}-I)
\end{align}
where $T$ is the shift operator $T[f](x) = f(x+1)$ and $I$ is the identity operator. 
The Krawtchouk polynomials are known to be polynomials eigenfunctions of $\cL $:
\begin{align}
 \cL K_n(x;p,N) = \lambda_n K_n(x;p,N) 
\end{align}
with  $\lambda_n = -n$.
We shall henceforth call $\cL$ the Krawtchouk operator.\par 
For later use, we shall introduce all the possible eigenfunctions of the Krawtchouk operator $\cL$ of the quasi-polynomial form: 
\begin{align}
\begin{split}
&\cL \phi (x)=\lambda \phi (x),\\
&\phi (x)=\xi (x)P_n(x)\quad (P_n(x)\in \mathbb{R}[x]).
\end{split}
\end{align}
It is readily seen that $P_n(x)$ satisfy
\begin{equation}
(\xi (x))^{-1}\cL \xi (x)P_n(x)=\lambda P_n(x),
\end{equation}
which implies that $P_n(x)$ is a polynomial eigenfunction of the transformed operator $\tilde{\cL}=(\xi (x))^{-1}\cL \xi (x)$. 
Here we choose $\xi (x)$ so that the transformed operator $\tilde{\cL}$ is again a Krawtchouk operator with different parameters:
\begin{equation}
\tilde{\cL}=\cL\big|_{x=\tilde{x},p=\tilde{p},N=\tilde{N}},
\end{equation} 
which is related with shape invariance in quantum mechanics \cite{odake2010another}.
From the ensuing necessary conditions on $\xi(x)$, we can then obtain the following four kinds of eigen-pairs $(\lambda, \phi(x))$:
\begin{eqnarray}
 (\lambda, \phi(x))  \in \{(\lambda^{(j)}_n, \phi^{(j)}_n(x)) \mid j\in \{1,2,3,4\}, n \in \dZ_{\ge 0}\},
\end{eqnarray}
where $\lambda_n^{(j)}$ and $\phi_n^{(j)}=\xi (x)P_n^{(j)}(x)$ are given as follows. 
\begin{align}
\begin{split}
&(\lambda^{(j)}_n,\xi^{(j)}(x),P_n^{(j)}(x))\\
&=\begin{cases}
(-n,1,K_n(x;p,N)) & (j=1)\\
(-N-n-1,(x-N)_{N+1},K_n(x-N-1;p,-N-2)) & (j=2)\\
(-N+n,(p-1)^xp^{-x},K_n(x;1-p,N)) & (j=3)\\
(n+1, (p-1)^{x}p^{-x} (x-N)_{N+1},K_n(x-N-1;1-p,-N-2)) & (j=4)
\end{cases}
\end{split}
\end{align}
Here, the possible range of $\lambda$ is that of the integers and 
for each integer $\lambda$, depending on the subset to which $n$ belongs, two cases are related: 
\begin{align}
\lambda = - n =
\left\{
\begin{array}{lc}
 \lambda^{(1)}_{n}   = \lambda^{(2)}_{n-N-1} & (N<n )\\
 \lambda^{(1)}_{n}   = \lambda^{(3)}_{N-n}   & (0\leq n \leq N)\\
 \lambda^{(3)}_{N-n} = \lambda^{(4)}_{-n-1}  & (n < 0)\\
\end{array}
\right..
\end{align}
The eigenvalues $\lambda_n^{(1)}$  and $\lambda_n^{(2)}$ are the mirror transformations of $\lambda_n^{(4)}$ and $\lambda_n^{(3)}$ with respect to $1/2$ and $-N+1/2$, respectively.
In the case of Meixner polynomials, this doubling does not occur in general because the constant corresponding to $N$ is real.
It is because of the symmetries of the Krawtchouk polynomials that these situations are possible.
Using the Diophantine property \eqref{K:factorprop}, one finds the simple expressions for the eigenfunctions:
\begin{align}
 \phi_n^{(j)}(x)
 =
 \begin{cases} K_n(x;p,N) & (j=1),\\
 (x-N)_{N+1}K_n(x-N-1;p,-N-2)=K_{n+N+1}(x;p,N) & (j=2)\\
 p^{-x}(p-1)^x K_n(x;1-p,N) & (j=3)\\
 (x-N)_{N+1} (1-p^{-1})^x K_n(x;1-p,N)=(1-p^{-1})^x K_{n+N+1}(x;1-p,N) & (j=4)
 \end{cases}.
\label{eigen:phi}
\end{align}
From the expressions, one finds that
\begin{align}
 &\phi_n^{(2)}(x)=\phi_{n+N+1}^{(1)}(x), 
\end{align}
and also from \eqref{K:dual} that 
\begin{align}
\begin{split}
 &\phi^{(3)}_{n}(x)= p^{-x}(p-1)^x(-1)^n\phi_{n}^{(1)}(N-x), \\
 &\phi^{(4)}_{n}(x)=\phi_{n+N+1}^{(3)}(x)= p^{-x}(p-1)^x(-1)^{n+N+1}\phi_{n+N+1}^{(1)}(N-x). 
 \end{split}
\end{align}
%

\section{Exceptional Krawtchouk operator}\label{sec3}
In this section, we introduce the $X$-Krawtchouk operator which is derived as follows through a Darboux transformation.
Let $(\mu,\chi(x))$ be an eigen-pair of the Krawtchouk operator $\cL$ such that $\cL \chi(x) = \mu \chi(x)$.
Then, the Krawtchouk operator $\cL$ can be factored as
\begin{align}
 \cL = \cB \circ \cF + \mu,
\end{align} 
where 
\begin{align}\label{eK:forward_backward}
\begin{split}
 &\cF = (\eta(x))^{-1}\left(\chi(x) T - \chi(x+1) I\right), \\ 
 &\cB = {(\chi(x))^{-1}}\left(p(N-x) I- x(1-p) T^{-1}\right){\eta(x)}, 
 \end{split}
\end{align}
with an arbitrary decoupling factor $\eta(x)$.
We define the new operator $\hat{\cL}$ by
\begin{equation}
\hat{\cL}= \cF \circ \cB.
\end{equation}
This $\hat{\cL}$ is a second order difference operator which is different from the Krawtchouk operator $\cL$. 
We will call $\hat{\cL }$ the exceptional Krawtchouk ($X$-Krawtchouk) operator. 
It should be remarked here that the eigen-pairs of the $X$-Krawtchouk operator $\hat{\cL}$ are given in terms of the eigen-pairs of the Krawtchouk operator such that $\cL \phi = \lambda \phi $:
\begin{equation}
\hat{\cL}\hat{\phi }=\hat{\lambda }\hat{\phi}
\end{equation} 
with
\begin{equation}
\hat{\phi }= \cF \phi ,\quad \hat{\lambda }=\lambda -\mu .
\end{equation}
The eigenfunction $\chi(x)$ plays the role of the seed function of the Darboux transformation. 
Here and hereafter we shall choose $\chi (x)=\phi^{(j)}_d(x)= \xi^{(j)}(x) \,P_d^{(j)}(x)$ (and $\mu =\lambda_d^{(j)}$) with $\phi^{(j)}_d(x)$ and $\xi^{(j)}(x)$ as in the previous section. The corresponding operators \eqref{eK:forward_backward} are explicitly given by
\begin{align}
\begin{split}
 &\cF^{(j,d)} = (\eta(x))^{-1}\left(\phi_d^{(j)}(x) T - \phi_d^{(j)}(x+1)I\right),\\
 &\cB^{(j,d)} = {(\phi_d^{(j)}(x))^{-1}}\left({p(N-x)}I- {x(1-p)} T^{-1}\right) {\eta(x)},
 \end{split}
\end{align} 
where $\eta(x)$ will be determined in the next section.

We write the exceptional Krawtchouk operator by 
\begin{align}\label{xoperator}
 \cL^{(j,d)} = \cF^{(j,d)} \circ \cB^{(j,d)} + \lambda_d^{(j)} I
\end{align}
for $j \in \{1,2,3,4\}$ and $n,d \in \dZ_{\ge 0}$.

From the discussion above, we find that the solutions to the eigenvalue problem
\begin{align}
\label{Ljd:eigen}
 \cL^{(j,d)} \left[\psi_{\ell,n}^{(j,d)}(x)\right] = \lambda_n^{(\ell)} \psi_{\ell,n}^{(j,d)}(x).
\end{align} 
are formally given by
\begin{align}
 \psi_{\ell,n}^{(j,d)}(x) = \cF^{(j,d)} [\phi_n^{(\ell)} (x)],
\end{align}
if $\cF^{(j,d)} [\phi_n^{(\ell)} (x)] \ne 0$ and otherwise (if $\cF^{(j,d)} [\phi_n^{(\ell)} (x)] = 0$), $\psi_{\ell,n}^{(j,d)}(x)$ is taken from $\mbox{Ker}(\cF^{(j,d)} \circ\cB^{(j,d)}) =\{f(x) \mid \cF^{(j,d)} \circ\cB^{(j,d)}[f(x)]=0\}.
$

\section{Exceptional Krawtchouk polynomials}\label{sec4}
In the previous section, we discussed the four types of seed functions for the Darboux transformation.
In order to construct the exceptional Krawtchouk polynomials $\hat K_n^{(j,d)}(x)$, we have to consider the Darboux transformation of the Krawtchouk polynomials, that is,  $\cF^{(j,d)}[\phi_n^{(1)}(x)]=\cF^{(j,d)}[K_n(x)]$. If $\cF^{(j,d)}[K_n(x)]\ne 0$, it is given as follows in terms of the second-order Casorati determinant:
\begin{align} \label{Xkraw-Darboux}
\begin{split}
\hat K_n^{(j,d)}(x)&=\hat K_n^{(j,d)}(x;p,N)=\dfrac{\cF^{(j,d)}[K_n(x)]}{\nu_n}\\
&=
\dfrac{1}{\nu_n\eta(x)}
\left|\begin{array}{ll}
\phi_d^{(j)}(x) & K_n(x)\\
\phi_d^{(j)}(x+1) & K_n(x+1)\\
\end{array}\right|,
\end{split}
\end{align}
where $\nu_n$ and $\eta(x)$ are chosen so that the eigenfunctions are monic polynomials by removing the common factors for any $n \in \mathbb{Z}_{\ge 0}$.
It is straightforward to see from (\ref{eigen:phi}) that  
\begin{align}
\begin{split}
& \hat K_n^{(1,d+N+1)}(x) = \hat K_n^{(2,d)}(x),\\ 
& \hat K_n^{(3,d+N+1)}(x) = \hat K_n^{(4,d)}(x).
\end{split}
\end{align}
Thus, in the rest of this paper, we shall restrict $d$ to be in the range $0 \le d \le N$ when $j=1$ or $3$.
For each of the cases $j=1$ to $4$, one can verify that 
\begin{align}
\begin{split}
& \left|\begin{array}{ll}
K_{d}(x) & K_n(x)\\
K_{d}(x+1) & K_n(x+1)\\
\end{array}\right| = (n-d) x^{n+d-1} + \cdots,\\
& \left|\begin{array}{ll}
\phi^{(2)}_d(x) & K_n(x)\\
\phi^{(2)}_d(x+1) & K_n(x+1)\\
\end{array}\right| = (n-d-N-1) (x-N+1)_{N}(x^{n+d} + \cdots),\\
 & \left|\begin{array}{ll}
\phi^{(3)}_d(x) & K_n(x)\\
\phi^{(3)}_d(x+1) & K_n(x+1)\\
\end{array}\right| 
= \dfrac{(p-1)^{x}}{p^{x+1}} (x^{n+d}+\cdots),
\\
 & \left|\begin{array}{ll}
\phi^{(4)}_d(x) & K_n(x)\\
\phi^{(4)}_d(x+1) & K_n(x+1)\\
\end{array}\right| 
 =  \dfrac{(p-1)^x}{p^{x+1}} (x-N+1)_N (x^{n+d+1}+\cdots),
\end{split}
\end{align}
which results in  
\begin{align}
\begin{split}
&\nu_n = \nu_n^{(j,d)} = 
\begin{cases}
  d-n & (j=1)\\
  d-n+N+1 & (j=2)\\
  1 & (j=3,4)\\
 \end{cases},\\
&\eta(x)=\eta^{(j)}(x)=
\begin{cases}
 -1 =- \xi^{(1)}(x) & (j=1)\\
 -(x-N+1)_N = (N-x)^{-1}\xi^{(2)}(x)& (j=2)\\
 p^{-x-1}(p-1)^x =p^{-1}\xi^{(3)}(x)& (j=3)\\
p^{-x-1}(p-1)^x (x-N+1)_N = p^{-1}(x-N)^{-1}\xi^{(4)}(x)  & (j=4)
 \end{cases}.
 \end{split}
\end{align}
In these cases, the degree of $\hat{K}_n^{(j,d)}(x)$ is given by
\begin{align}
 \deg [\hat{K}_n^{(j,d)}(x)] = 
\left\{\begin{array}{ll}
n+d-1 &  (j=1)\\ 
n+d &  (j =2 \mbox{ or }3 )\\ 
n+d+1 &  (j=4)\\ 
\end{array}\right..
\end{align}
Note that the cases where $\nu_n^{(1,d)}=0$ and $\nu^{(2,d)}_n=0$ need to be discussed separately.
\par
For cases 3 and 4, observe that the relations below are satisfied on the grid points $x \in \{0,1,2,\ldots,N-1\}$, where the orthogonality holds:
\begin{subequations}
\begin{align}
&\hat K_{n}^{(3,d)}(x) = \gamma_{n,d,N} \hat K_{n}^{(1,N-d)}(x), \label{Diophan:1-3a}\\
&\hat K_{n}^{(4,d)}(x) = \gamma_{n,d,N}^* p^{-x} (p-1)^x\hat K_{N-n}^{(2,d)}(N-x-1),\label{Diophan:2-4a}
\end{align}
\end{subequations}
with $\gamma_{n,d,N}= \frac{ p^{d-N}(p-1)^d (-N)_d}{(-1)^d(-N)_{N-d}}(n+d-N)$ and $\gamma_{n,d,N}^*= \frac{p^n (p-1)^{n-N+1}(-N)_n}{(-1)^{d+1}(-N)_{N-n}}(n+d+1)$.
Therefore, cases 3 and 4 can be obtained from $\hat K_n^{(1,d)}(x)$ and $\hat K_n^{(2,d)}(x)$.\par

From these results, we have obtained $\hat{K}^{(j,d)}_n(x)$ as the $X$-Krawtchouk polynomials which are polynomial eigenfunctions of the $X$-Krawtchouk operator defined from \eqref{xoperator}. 
Since the $\hat{K}_{n}^{(j,d)}(x)$ are defined by considering $\cF^{(j,d)}[\phi^{(1)}_n(x)]$, the other $X$-Krawtchouk polynomials are expected to be obtained from $\cF^{(j,d)}[\phi^{(\ell)}_n(x)]~ (\ell=2,3,4)$. However, if $\cF^{(j,d)}[\phi^{(\ell)}_n(x)]\ne 0$, these functions take zero values at $x=1,2,\ldots ,N$ when $\ell=2$ and do not give polynomial eigenfunctions in $x$ when $\ell=3,4$. 
Therefore, in the case of $\cF^{(j,d)}[\phi^{(\ell)}_n(x)]\ne 0$, $\hat{K}^{(j,d)}_n(x)$ are the only $X$-Krawtchouk polynomials.
It is straightforward to see that the $X$-Krawtchouk polynomials $\hat{K}^{(j,d)}_n(x)$ are polynomial eigenfunctions of the $X$-Krawtchouk operator:
\begin{equation}
\cF^{(j,d)}\circ \cB^{(j,d)}\hat{K}^{(j,d)}_n(x)= (-n-\lambda_d^{(j)})\hat{K}^{(j,d)}_n(x).
\end{equation}
One can also find other polynomial eigenfunctions to this problem when $\cF^{(j,d)}[\phi^{(\ell)}_n(x)]= 0$. We have then the following cases:
\begin{align}
\begin{split}
\{ (j,d,\ell,n) \mid \cF^{(j,d)}[\phi^{(\ell)}_n(x)]= 0\}
=\{ (j,d,j,d)~|~1\le j\le 4,d\ge 0\}\\
\qquad \cup \{ (2,d,1,d+N+1) \mid d\ge 0\}\\
\qquad \cup \{ (4,d,3,d+N+1) \mid d\ge 0\}.
\end{split}
\end{align}
When $(\ell,n)=(j,d)$, we only have to consider $\mbox{Ker}(\cB^{(j,d)})$, i.e. $\cB^{(j,d)} \psi(x)=0$, from where we obtain 
\begin{align}
 \psi_{j,d}^{(j,d)}(x) =
\left\{\begin{array}{cc}
(-1)^{x+N}(1-p)^x p^{-x}(-x)_{N}  &  (j=1)\\
 (-1)^x (1-p)^x p^{-x}  &  (j=2)\\
(-1)^N (-x)_{N} &  (j=3)\\
 1  &  (j=4)
\end{array}\right..
\label{Ker_B}
\end{align}
We can determine that $\psi_{4,d}^{(4,d)}(x)=1$ belongs to the class of $X$-Krawtchouk polynomials $\{\hat K_{n}^{(4,d)}(x)\}$ and we then introduce
\begin{align}
\hat K_{-d-1}^{(4,d)}(x)=\psi_{4,d}^{(4,d)}=1,
\end{align}
where the subscript of $\hat K_{-d-1}^{(4,d)}(x)$ is chosen to keep $\deg \hat K_n^{(4,d)}= n+d+1$. 
In the case of $\psi_{3,d}^{(3,d)}(x)$, it should be remarked that $\psi_{3,d}^{(3,d)}(x)=\hat{K}_{N-d}^{(j,d)}(x)\propto \psi_{1,N-d}^{(3,d)}(x)$.
In the other cases, we should directly solve $\cF^{(j,d)}\circ \cB^{(j,d)}\psi =0$ to find
\begin{align}\label{x-kraw2:N+d+1}
\begin{split}
 \hat{K}_{N+d+1}^{(2,d)}(x) &=\psi_{1,N+d+1}^{(2,d)}= \hat{K}_{N+d+1}^{(2,d)}(x;p,N)\\
&=\sum_{0 \leq k,j \leq d}
\sum_{\ell=0}^{N+k+j+1}  
(-1)^{k+j+\ell} \dfrac{ (-d)_j(-d)_k \,(1-p)^{d-k}\,(-p)^{N+d+1+k-\ell}}{\, j! \,k!\,\ell! }\\
 &\times  (N+j+k+1)!\,(-N-d-1)_{d-j} (-N-d-1)_{d-k}{(-x-k-1)_{\ell}}\\
&=\sum_{0 \le k,j \le d} \frac{(-d)_j(-d)_k(p-1)^{d-k}p^{d-j}}{j!k!}
(-N-d-1)_{d-j}(-N-d-1)_{d-k} \\
&\times K_{N+k+j+1}(x+k+1;p,N+k+j+1),
 \end{split}
\end{align}
so that $\cB^{(2,d)}[\hat K^{(2,d)}_{N+d+1}(x)] = K_{N+d+1}(x) \in \mbox{Ker}(\cF^{(2,d)})$.
In a similar way, the functions $\psi_{3,N+d+1}^{(4,d)}$ can also be calculated although they are not polynomials in $x$.
It should be remarked here that $\hat{K}_{N+d+1}^{(2,d)}(x)$ can formally be obtained as follows:
\begin{equation}
\hat{K}_{N+d+1}^{(2,d)}(x)=\lim_{\varepsilon \to 0} \hat{K}_{N+d+1}^{(2,d)}(x;p,N+\varepsilon) .
\end{equation}
In addition, the Krawtchouk polynomials can be recovered by acting with $\cB^{(j,d)}$ on the $X$-Krawtchouk polynomials $\hat K_n^{(j,d)}(x)$ so as to find
\begin{align}
 \cB^{(j,d)}[\hat K_n^{(j,d)}(x)]
 = \tilde \nu_n^{(j,d)}K_n(x),
\label{Bop:XKtoK}
\end{align}
with $\tilde \nu_n^{(j,d)} = \delta_{1,j} + \delta_{2,j} + (\delta_{3,j}+\delta_{4,j})(\lambda_n - \lambda_d^{(j)})$,
except in the cases where $\tilde \nu_n^{(j,d)}=0$, i.e.
\begin{align}
 \cB^{(3,d)}[\hat K_{N-d}^{(3,d)}(x)] =  \cB^{(4,d)}[\hat K_{-d-1}^{(4,d)}(x)] = 0.
\label{Bop:XKtoZero}
\end{align}

\section{Properties of the exceptional Krawtchouk polynomials}\label{sec5}
We have introduced all classes of $X$-Krawtchouk polynomials $\{\hat{K}_{n}^{(j,d)}\}$ by means of a Darboux transformation.
These $\hat{K}_{n}^{(j,d)}$ are sometimes called the $X$-Krawtchouk polynomials of type $j$. 
In this section, we shall examine their properties.

\subsection{Diophantine property/factorization}
Analogues of the Diophantine property of the ordinary Krawtchouk polynomials \eqref{K:factorprop} can be found for the $X$-Krawtchouk polynomials and will be called the Diophantine properties of the $X$-Krawtchouk polynomials. They are listed below.

When $n > N$, we have
\begin{subequations} \label{Diophan:first}
\begin{align}
& \hat{K}_n^{(1,d)}(x) = (x-N+1)_{N} \hat{K}_{n-N-1}^{(2,d)}(x-N-1;p,-N-2) \quad (n \ne d), \label{Diophan:1-2}\\
& \hat{K}_n^{(2,d)}(x) = (x-N)_{N+2} \hat{K}_{n-N-1}^{(1,d)}(x-N-1;p,-N-2) \quad (n \ne N+d+1),\\
& \hat{K}_n^{(3,d)}(x) = (x-N+1)_{N} \hat{K}_{n-N-1}^{(4,d)}(x-N-1;p,-N-2),\label{Diophan:3-4}\\
& \hat{K}_n^{(4,d)}(x) = (x-N)_{N+2} \hat{K}_{n-N-1}^{(3,d)}(x-N-1;p,-N-2), \label{Diophan:4-3}
\end{align}
\end{subequations}
and
\begin{align}
 \hat{K}_{N-d}^{(3,d)}(x) = (x-N+1)_{N}.
\end{align}
When $d> N $, we have 
\begin{align}
\begin{split}
& \hat{K}_n^{(1,d)}(x) = (x-N+1)_{N} \,\hat{K}_{n}^{(2,d-N-1)}(x), \\
& \hat{K}_n^{(3,d)}(x) = (x-N+1)_{N} \,\hat{K}_{n}^{(4,d-N-1)}(x).
\end{split}
\end{align}
When $n,d > N$, we have  
\begin{align}
\begin{split}
 \hat{K}_n^{(1,d)}(x) 
&= (-x)_{N} \,(-1-x)_{N+2} \,\hat{K}_{n-N-1}^{(1,d-N-1)}(x-N-1;p,-N-2) \quad (n \ne d), \\
\hat{K}_n^{(3,d)}(x) 
&= (-x)_{N} \,(-1-x)_{N+2} \,\hat{K}_{n-N-1}^{(3,d-N-1)}(x-N-1;p,-N-2),
\end{split}
\end{align}
where we have used that for $m\ge 0$ and $d>N$,
\begin{align}
\begin{split}\label{Diophan:last}
 \hat{K}_{m}^{(2,d)}(x;p,-N-2) &=(x+1)_{N+2} \,\hat{K}_{m}^{(1,d-N-1)}(x;p,-N-2),\\
 \hat{K}_{m}^{(4,d)}(x;p,-N-2) &=(x+1)_{N+2} \,\hat{K}_{m}^{(3,d-N-1)}(x;p,-N-2).
 \end{split}
\end{align}

\subsection{Orthogonality}
\subsubsection{Type 1 and Type 3}
From the Diophantine property \eqref{Diophan:1-2}, one finds that 
for $x \in \{0,1,2,\ldots,N-1\}$ $\hat{K}_{n}^{(1,d)}(x)$ can take a non-zero value only when $n,d \in \{0,1,2,\ldots,N\}$.
Suppose that $d \in \{0,1,2,\ldots,N\}$ and let $X_{(1,d)} = \{0,1,\ldots,d-1,d+1,\ldots,N\}$.
In this case, the so-called state-deletion occurs.
The set $\{\hat{K}_n^{(1,d)}(x) \mid n \in X_{(1,d)} \}$ has the discrete orthogonality:
\begin{equation}
 \sum_{x =0}^{N-1}  \hat w^{(1,d)}(x) \hat{K}_n^{(1,d)}(x) \hat{K}_m^{(1,d)}(x) = \hat h_n^{(1,d)} \delta_{n,m}
\end{equation}
for $n,m \in X_{(1,d)}$, where
\begin{align}
\begin{split}
  \hat w^{(1,d)}(x;p,N) &= \dfrac{N-x}{N(1-p)} \dfrac{w(x;p,N)}{P_d^{(1)}(x) P_d^{(1)}(x+1)}
=
\dfrac{w(x;p,N-1)}{P_d^{(1)}(x) P_d^{(1)}(x+1)}, \\
\hat h^{(1,d)}_n &= \frac{h_n}{(\lambda_d^{(1)}-\lambda_n)Np(1-p)}.
\end{split}
\end{align}

Only when $d=0,N$, does the weight function $\hat w^{(1,d)}(x;p,N)$ not change sign on $x \in \{ 0,1,\ldots ,N-1\} $. This implies that the exceptional Krawtchouk polynomials of type 1 $\{ \hat{K}_{n}^{(1,d)}\}_{n\in X_{(1,d)}}$ are positive definite when $d=0,N$.

In the same fashion, we observe from \eqref{Diophan:3-4} that the $X$-Krawtchouk polynomials of type 3 $\hat{K}_{n}^{(3,d)}$ on $x\in \{ 0,1,\ldots ,N-1\}$ take a non-zero value when $n\in X_{(3,d)} = \{0,1,\ldots,N-d-1,N-d+1,\ldots,N\}$. Furthermore, from \eqref{Diophan:1-3a} we see that the orthogonality for the set $\{ X_n^{(3,d)}\}_{n\in X_{(3,d)}}$ is essentially the same as that for the $X$-Krawtchouk polynomials of type 1.

\subsubsection{Type 2 and type 4}\label{sec:orthogonality_24}
In this case, a state-addition occurs.
Let $X_{(2,d)} = \{0,1,\ldots,N, N+d+1\}$.
The set $\{\hat{K}_n^{(2,d)}(x) \mid n \in X_{(2,d)} \}$ has the discrete orthogonality:
\begin{equation}
 \sum_{x=-1}^{N}  \hat w^{(2,d)}(x) \hat{K}_n^{(2,d)}(x) \hat{K}_m^{(2,d)}(x) = \hat h_n^{(2,d)} \delta_{n,m}
\end{equation}
for $n,m \in X_{(2,d)}$, where
\begin{align}
\begin{split}
  \hat w^{(2,d)}(x) &= \dfrac{p(N+1)}{(x+1)} \dfrac{w(x;p,N)}{P_d^{(2)}(x)P_d^{(2)}(x+1)} 
= \dfrac{w(x+1;p,N+1) }{P_d^{(2)}(x)P_d^{(2)}(x+1)} , \\
  \hat h^{(2,d)}_n &= \frac{(N+1)h_n}{\lambda_n-\lambda_d^{(2)}}
= (-1)^{d} \,d!\, (N+1)! \,(N+d+1)!\, (1-p)^{N+d+1} p^{N+d+1}.
\end{split}
\end{align}
When $d$ is an even integer, one finds that the weight function $\hat w^{(2,d)}(x)$ is positive on $x\in X_{(2,d)}$.


As in the cases of type $1$ and of type $3$ $X$-Krawtchouk polynomials, we find from \eqref{Diophan:4-3} that the $X$-Krawtchouk polynomials of type 4 $\hat{K}_{n}^{(4,d)}$ on $x\in \{ -1,0,1,\ldots ,N\} $ take a non-zero value when $n\in X_{(4,d)} = \{-d-1,0,1,\ldots,N\}$. Furthermore, from \eqref{Diophan:2-4a} we see that the orthogonality for $\{ X_n^{(4,d)}\}_{n\in X_{(4,d)}}$ is essentially the same as that of the $X$-Krawtchouk polynomials of type 2.


\subsection{Recurrence relations}
The ordinary orthogonal polynomials are known to satisfy the three term recurrence relations. However XOPs satisfy the recurrence relations with more terms \cite{miki2015new}.
In order to derive the recurrence formula of the $X$-Krawtchouk polynomials, we first give the resultant of the Krawtchouk polynomials \cite{ismail2005discriminant}.
\begin{lemma}
Let us denote the resultant of $f(x),g(x)\in \mathbb{R}[x]$ by $\mathrm{Res}(f(x),g(x))$. 
For $n=1,2,\ldots$ and $a\in \mathbb{R}$, the following relation holds:
\begin{align}\label{resultant_12}
\begin{split}
& \mathrm{Res}\left(K_n(x;p,a),K_n(x+1;p,a)\right)\\
& \qquad =  (-n)^n \mathrm{Res}\left(K_{n-1}(x;p,a-1),K_n(x+1;p,a)\right), \\
& \mathrm{Res}\left(K_n(x;p,a-1),K_{n+1}(x+1;p,a)\right)\\
& \qquad = \left((n-a)np(1-p)\right)^n \mathrm{Res}\left(K_{n-1}(x;p,a-1),K_n(x+1;p,a)\right)
\end{split}
\end{align}
and
\begin{align}
\begin{split}
& \mathrm{Res}\left(K_n(x;p,a),K_n(x+1;p,a)\right)\\
& \qquad =  n^n \prod_{k=1}^{n-1} k^k (k-a)^k p^{n(n-1)/2}(1-p)^{n(n-1)/2}.
\end{split}
\label{resultant_K_K}
\end{align}
\end{lemma}
\begin{proof}
Recall that $K_n(x;p,a)=(-a)_n(p)^n\hg{2}{1}\argu{\!-n,-x\!}{-a}{\dfrac{1}{p}}$. 
By using the contiguous relations of the hypergeometric functions \cite{bateman1953higher}:
\begin{align}
\begin{split}
& \hg{2}{1}\argu{\!\!-n,1-x\!\!}{-a}{\frac{1}{p}} - \hg{2}{1}\argu{\!\!-n,-x\!\!}{-a}{\frac{1}{p}}= \dfrac{n}{ap}\,\hg{2}{1}\argu{\!1-n,1-x\!}{1-a}{\frac{1}{p}} ,\\
& a p \hg{2}{1}\argu{\!\!-n-1,-x\!\!}{-a}{\frac{1}{p}} +x\hg{2}{1}\argu{\!\!-n,1-x\!\!}{1-a}{\frac{1}{p}} = ap \hg{2}{1}\argu{\!-n,-x\!}{-a}{\frac{1}{p}},\\
& a \hg{2}{1}\argu{\!\!-n,-x\!\!}{-a}{\frac{1}{p}} +(n-a)\hg{2}{1}\argu{\!\!-n,1-x\!\!}{1-a}{\frac{1}{p}} = n\left(1-\dfrac{1}{p}\right) \hg{2}{1}\argu{\!1-n,1-x\!}{1-a}{\frac{1}{p}},
\end{split}
\end{align}
and $\mathrm{Res}(K_1(x;p, a), 1)=1$,
simple calculations show that \eqref{resultant_12} holds and 
\eqref{resultant_K_K} can be obtained by induction.
\end{proof}
\begin{corollary}\label{cor:zero}
$K_n(x;p,a)$ and  $K_n(x+1;p,a)$ have no common zeros if and only if $p \notin \{0,1\}$ and $a \notin \{1,2,\ldots,n-1\}$.
\end{corollary}

\begin{proposition}\label{prop53}
For $q(x) \in \mathbb{R}[x]$, $\cB^{(j,d)}[q(x)]\in \dR[x]$ if and only if 
\begin{equation}
q(x) \in \mbox{\rm span}\{\hat{K}_n^{(j,d)}(x)\}_{n\in \dZ_{\ge 0}^{(j,d)}},
\end{equation} 
where 
\begin{equation}
\dZ_{\ge 0}^{(j,d)}=\begin{cases}
\dZ_{\ge 0} \backslash \{d\} & (j=1)\\
\dZ_{\ge 0} & (j=2,3)\\
\dZ_{\ge 0} \cup \{-d-1\} & (j=4)
\end{cases}.
\end{equation}
\end{proposition}
\begin{proof}
Assume that $\cB^{(j,d)}[q(x)]\in \dR[x]$ and write $r(x)=\cB^{(j,d)}[q(x)]$. We have
\begin{align}\label{proof:eq1}
\begin{split}
\left\{
\begin{array}{ll}
 p \,(x-N)\, q(x) + (1-p) \,x \,q(x-1) = P_d^{(1)}(x)\,r(x) & \mbox{(if $j=1$)}\\
 p \,q(x) + (1-p) \, q(x-1) = P_d^{(2)}(x)\,r(x) & \mbox{(if $j=2$)}\\
 (N-x) \,q(x) + x \, q(x-1) = P_d^{(3)}(x)\,r(x) & \mbox{(if $j=3$)}\\
 -\,q(x) +  \, q(x-1) = P_d^{(4)}(x)\,r(x) & \mbox{(if $j=4$)}
\end{array}\right. .
\end{split}
\end{align}
If $r(x)$ can be expanded in terms of the Krawtchouk polynomials $\{ K_n(x)\}_{n\in \mathbb{Z}_{\ge 0}}$ as follows
\begin{align}
 r(x)=\sum_{\substack{n =0 \\ n \in \dZ_{\ge 0}^{(j,d)}}}^{m} \tilde \nu_n^{(j,d)} c_n^{(j)} K_n(x) 
\label{expand:r}
\end{align}
with $m= \deg r(x)$,
we can find a particular solution to \eqref{proof:eq1} by  using \eqref{Bop:XKtoK} and \eqref{Bop:XKtoZero}:
\begin{align}
\begin{split}
q(x)= q_{r}^{(j)}(x) &= \sum_{\substack{n =0 \\ n \in \dZ_{\ge 0}^{(j,d)}}}^{m} c_n^{(j)} \hat K_n^{(j,d)}(x) 
\in \mbox{span} \{\hat K_n^{(j,d)}(x)\}.
\end{split}
\end{align}
Note here that if $\{c_n^{(j)}\}$ are arbitrary constants, then $r(x)$ is an arbitrary polynomial of degree $m$ when $j=2$ and $j=4$. 
But if $q(x)$ is expanded in terms of the Krawtchouk polynomial and contains $K_d(x)$ when $j=1$, or $K_{N-d}(x)$ when $j=3$, then the corresponding polynomial $q(x)$ does not exist and the assumption cannot be satisfied. Hence we can see that by excluding $K_d(x)$ when $j=1$ and $K_{N-d}(x)$ when $j=3$ from the linear combination of sequence of the Krawtchouk polynomials, $r(x)$ of the form \eqref{expand:r} is a general polynomial of degree $m$ that can satisfy the assumption \eqref{proof:eq1} for a given $(j,d)$.

Then, by setting $q(x)=q_{0}^{(j)}(x)+q_{r}^{(j)}(x)$, the equation of $q_{0}^{(j)}(x)$ becomes a homogeneous first-order difference equation with $r(x)=0$ in \eqref{proof:eq1}, and we see that $q_{0}^{(j)}(x)$ can be obtained from 
$\mbox{Ker}(\cB^{(j,d)}) \cup \dR[x]$. 
From \eqref{Ker_B}, we obtain
\begin{align}
q_{0}^{(1)}(x) = q_{0}^{(2)}(x) = 0, \quad
q_{0}^{(3)}(x) = \alpha_3 \hat K^{(3,d)}_{N-d}(x), \quad
q_{0}^{(4)}(x) = \alpha_4 \hat K_{-d-1}^{(4,d)}(x),
\end{align}
where $\alpha_3$ and $\alpha_4$ are arbitrary constants, and  
the general solution of \eqref{proof:eq1} 
in the following form:
\begin{eqnarray*}
 q(x) = q_{0}^{(j)}(x) + q_{r}^{(j)}(x).
\end{eqnarray*}
Hence we conclude that $q(x) \in \mbox{\rm span}\{\hat{K}_n^{(j,d)}(x)\}_{n\in \dZ_{\ge 0}^{(j,d)}}$. 
Sufficiency is obvious from \eqref{Bop:XKtoK} and \eqref{Bop:XKtoZero}.
\end{proof}

\begin{proposition}\label{prop54}
For $\pi(x) \in \mathbb{R}[x]$, $\cL^{(j,d)}[\pi(x)] \in \dR[x]$ if and only if $\cB^{(j,d)} [\pi(x)] \in \dR[x]$.
\end{proposition}
\begin{proof}
Suppose that $\cB^{(j,d)} [\pi(x)]$ is not a polynomial. Then it is given by a rational function of the following form:
\begin{align}
 \cB^{(j,d)} [\pi(x)] = \pi_0(x) +\sum_{k=1}^{m} \dfrac{\pi_k(x)}{(x-a_k)^{\mu_k}},
\label{XBtoPi}
\end{align}
where $\mu_1,\ldots, \mu_m$ are positive integers  and $\pi_{0}(x), \ldots, \pi_{m}(x)$ are some polynomials with the conditions $\pi_k(a_{k}) \ne 0$ for $k=1,2,\ldots,m$.
We here note that $a_k$ are the zeros of $P_d^{(j)}(x)$, i.e. $P_d^{(j)}(a_k)=0$.
Then, by applying $\cF^{(j,d)}$ to \eqref{XBtoPi}, we obtain
\begin{align}
 \cL^{(j,d)} [\pi(x)] = \dfrac{\tilde \pi(x)}{\prod_{k=1}^{m} (x-a_k)^{\mu_k}(x+1-a_k)^{\mu_k}},
\end{align}
where
\begin{align}
\begin{split}
 \tilde \pi(x) =&  (x-a_k)^{\mu_k} (x+1-a_k)^{\mu_k} \pi^{*}(x)
 +\rho_{j}(x) (x-a_k)^{\mu_k} P_d^{(j)}(x) \pi_k(x+1)  \\
&+\rho_{j}^{*}(x) (x+1-a_k)^{\mu_k} P_d^{(j)}(x+1) \pi_k(x)
\end{split}
\end{align}
with $\pi^{*}(x)$ a polynomial and 
\begin{align*}
& \rho_{1}(x)=-1, &&\rho_{1}^{*}(x)=1,\\
& \rho_{2}(x)=N-x, &&\rho_{2}^{*}(x)=1+x,\\
& \rho_{3}(x)=p, &&\rho_{3}^{*}(x)=1-p,\\
& \rho_{4}(x)=p(x-N), &&\rho_{4}^{*}(x)=(1-p)(1+x).
\end{align*}
Using Corollary \ref{cor:zero}, it is shown that if $P_d^{(j)}(a_k) = 0$, then $P_d^{(j)}(a_k+1)\ne 0$ and $P_d^{(j)}(a_k-1) \ne 0$, and further that
$\tilde \pi(a_k)=0 $ and $\tilde \pi(a_k-1)=0$ cannot be true at the same time, since
\begin{align*}
\begin{array}{ll}
P_d^{(j)}(a_k+ 1) \pi_k(a_k) \ne 0,\quad P_d^{(j)}(a_k- 1) \pi_k(a_k) \ne 0
 & (j \in \{1,3\}),\\
(a_k+1)P_d^{(j)}(a_k+1) \pi_k(a_k) \ne 0,\quad   (a_k-N-1)P_d^{(j)}(a_k-1) \pi_k(a_k) \ne 0
 & (j\in \{2,4\}).
\end{array}
\end{align*}
Thus $\cL^{(j,d)}[\pi(x)]$ can never be polynomial if $\cB^{(j,d)}[\pi(x)]$ is not polynomial. Hence, if $\cL^{(j,d)} [\pi(x)]$ is a polynomial, then $\cB^{(j,d)}[\pi(x)]$ is a polynomial.

It is obvious that if $\cB^{(j,d)}[p(x)]$ is a polynomial, then $\cL^{(j,d)} [p(x)]$ is a polynomial.
\end{proof}

\begin{proposition}\label{prop:recurrence}
Let $q_{\pi}(x)$ be a non-constant polynomial such that
\begin{align}
q_{\pi}(x)  \in 
\begin{cases}
\mbox{\rm span} \{\hat K_{n}^{(4,d)}(x-N-1,1-p,-N-2)\}_{n\in \mathbb{Z}_{\ge 0}^{(4,d)}}  & (j=1)\\
\mbox{\rm span}\{\hat K_{n}^{(4,d)}(x,1-p,N)\}_{n\in \mathbb{Z}_{\ge 0}^{(4,d)}}  & (j=2)\\
\mbox{\rm span}\{\hat K_{n}^{(4,d)}(x-N-1,p,-N-2)\}_{n\in \mathbb{Z}_{\ge 0}^{(4,d)}}  & (j=3)\\
\mbox{\rm span}\{\hat K_{n}^{(4,d)}(x,p,N)\}_{n\in \mathbb{Z}_{\ge 0}^{(4,d)}}  & (j=4)
\end{cases}.
\label{rec:spectrum}
\end{align}
Then there exists a sequence $\{c_{n,\ell}^{(j,d)}\}_{n,\ell}$ satisfying
\begin{align}
 q_{\pi}(x) \hat{K}_n^{(j,d)}(x)
= \sum_{\substack{\ell=n-m\\ n+\ell\in\dZ_{\ge 0}^{(j,d)}}}^{n+m}c_{n,\ell}^{(j,d)}  \hat{K}_{\ell}^{(j,d)}(x),
\label{XK:RR}
\end{align}
where $m=\deg q_{\pi}(x) \ge d+1$.
In particular, for the lowest degree, $m=d+1$, $q_{\pi}(x)$ is given by 
\begin{align}\label{rec:spectrum_low}
 q_{\pi}(x)=
\begin{cases}
K_{d+1}(x+1;p,N+1) &(j=1)\\
K_{d+1}(x-N;p,-N-1) &(j=2)\\
K_{d+1}(x+1;1-p,N+1) &(j=3)\\
K_{d+1}(x-N;1-p,-N-1) &(j=4)\\
\end{cases}.
\end{align}
\end{proposition}
\begin{proof}  
Let us define
\begin{align}
  \pi (x)=p(N-x) \dfrac{\eta^{(j)}(x)}{\xi^{(j)}(x) P^{(j)}_d(x)}
\left( q_{\pi}(x) - q_{\pi}(x-1) \right) .
\label{def:pi}
\end{align}
From \eqref{rec:spectrum}, one finds that
$\pi(x)\in \dR[x]$ and
\begin{align}
\deg \pi(x) = m - d -\delta_{2,j}-\delta_{4,j}.
\end{align}
Note that for $j=1,3$, $P^{(j)}_d(x)$ does not have a factor of $x-N$ and then $\pi(x)$ must have a factor of $x-N$ in order to satisfy the assumption of this proposition.
Thus it follows that $m  \ge d+1$, since
$\deg \pi(x) \ge \delta_{1,j}+\delta_{3,j}$. 
By applying $\cB^{(j,d)}$ to $q_{\pi}(x) \hat K_n^{(j,d)}(x)$, we obtain 
\begin{align}
\begin{split}
& {\cB^{(j,d)}}[q_{\pi}(x) \hat K_n^{(j,d)}(x)]  \\
&  \qquad = \dfrac{p(N-x)\eta^{(j)}(x)}{\xi^{(j)}(x) P^{(j)}_d(x)}(I-T^{-1})[q_{\pi}(x)]\,\hat K_n^{(j,d)}(x) + 
q_{\pi}(x) {\cB^{(j,d)}}[\hat K_n^{(j,d)}(x)]  \\
&  \qquad = 
\pi(x)\hat K_n^{(j,d)}(x) + 
q_{\pi}(x) \tilde \nu_n^{(j,d)} K_n(x) 
\\
&  \qquad = 
 \sum_{\ell = n-m}^{n+m} c_{\ell}^{*}\,K_{\ell}(x)
\in \dR[x],
\end{split}
\label{expand:BqXK in K}
\end{align}
where we have used \eqref{K:TRR} and \eqref{K:rec_variant}.
With the help of Proposition \ref{prop53} and \ref{prop54}, $q_{\pi }(x) \hat K_n^{(j,d)}(x)$ can be presented as a linear combination of the exceptional Krawtchouk polynomial $\hat K_n^{(j,d)}$:
\begin{align}
 q_{\pi}(x) \hat{K}_n^{(j,d)}(x)
= \sum_{\ell} c_{n,\ell}^{(j,d)}\,  \hat{K}_{\ell}^{(j,d)}(x).
\label{qXK_n in K}
\end{align} 
By applying $\cB^{(j,d)}$ to both sides of \eqref{qXK_n in K}, we obtain
\begin{align}
 \cB^{(j,d)}[q_{\pi}(x) \hat{K}_n^{(j,d)}(x)]
= 
\sum_{\ell} c_{n,\ell}^{(j,d)}\,\tilde \nu_{\ell}^{(j,d)}\,  {K}_{\ell}(x),
\end{align} 
which can be considered together with \eqref{expand:BqXK in K} to define the range of the sum via
$c_{n,\ell}^{(j,d)} = 0$ for $ |\ell-n|>d+m$ if $\tilde \nu_n^{(j,d)} \ne 0$. 
Notice that $\tilde{\nu}^{(4,d)}_{-d-1}=\tilde{\nu }_{N-d}^{(3,d)}=0$ and thus we can not find $c_{n,-d-1}^{(4,d)}$ and $c_{n,N-d}^{(3,d)}$ from the above method.
We can determine $c_{n,-d-1}^{(4,d)}$ as follows. From the discussion in Sec. \ref{sec:orthogonality_24}, we recall the orthogonality relation of the type-4 $X$-Krawtchouk polynomials given by
\begin{equation}
\left< \hat{K}_{n}^{(4,d)}(x),\hat{K}_{n}^{(4,d)}(x)\right>_4 = \sum_{x=-1}^N \hat{w}^{(4,d)}(x)\hat{K}_{n}^{(4,d)}(x),\hat{K}_{n}^{(4,d)}(x) =\hat{h}_{n}^{(4,d)}\delta_{mn}
\end{equation} 
with $\hat{h}_n^{(4,d)}\ne 0$ for $m,n\in X_{(4,d)}$.
When $m-d-1< n\le N$, one sees from the definition that $q_{\pi}(x) \in \mbox{span}\{\hat K_{-d-1}^{(4,d)}, \hat K_{0}^{(4,d)},\ldots,\hat K_{m-d-1}^{(4,d)}\}$ and finds that 
\begin{align}
\left< q_{\pi }(x),\hat{K}_{n}^{(4,d)}(x)\right>_4=\left< 1, q_{\pi }(x)\hat{K}_{n}^{(4,d)}(x)\right>_4=\left< K_{-d-1}^{(4,d)}(x), q_{\pi }(x)\hat{K}_{n}^{(4,d)}(x)\right>_4=0.
\end{align}
Therefore, from \eqref{qXK_n in K}, one concludes that $c_{n,-d-1}^{(4,d)}=0~~(m-d-1<n\le N)$. For $n>N$, from the Diophantine property \eqref{Diophan:4-3} and \eqref{qXK_n in K}, one sees that the following relation holds:
\begin{align}
\begin{split}
 0& = q_{\pi}(k) \hat K_n^{(4,d)}(k) = \sum_{\ell \in X_{(4,d)}} c_{n,\ell}^{(4,d)} \hat K_{\ell}^{(4,d)}(k)\quad (k=-1,0,\ldots ,N).
\end{split}
\end{align}
From the linear independence of the $X$-Krwatchouk polynomials $\{ \hat{K}_n^{(4,d)}(x)\} _{n\in X_{(4,d)}}$ on $x=-1,0,\ldots ,N$, one finds that $c_{n,-d-1}^{(4,d)}=0~~(n>N)$.\\
With respect to $c_{n,N-d}^{(3,d)}$, from the Diophantine property \eqref{Diophan:3-4} and the recurrence relation of the $X$-Krawtchouk polynomials of type 4, we can immediately see that $c_{n,N-d}^{(3,d)}=0\quad (n>m+N-d)$. The relation $c_{n,N-d}^{(3,d)}=0\quad (n<-m+N-d)$ is also verified by comparing both sides of \eqref{qXK_n in K}. Finally, it is shown that \eqref{XK:RR} holds including for the cases where $\tilde \nu^{(j,d)}_{n}=0$.
Eq. \eqref{rec:spectrum_low} is almost obvious from \eqref{K:dual}, \eqref{K:rec_variant} and \eqref{Xkraw-Darboux}.
\end{proof}

If $\cF^{(j,d)}[K_n(x)] \ne 0$, one observes from \eqref{expand:BqXK in K} that
\begin{align}
\begin{split}
& \pi(x)\hat K_n^{(j,d)}(x) + q_{\pi}(x) \tilde \nu_n^{(j,d)} K_n(x) \\
&  \qquad = \tilde q_{m-1}^{(j)}(x) (N-x) K_{n}(x+1) + \tilde q_{m}^{(j)}(x) K_{n}(x)   + q_{\pi}(x) \tilde \nu_n^{(j,d)} K_n(x),
\end{split}
\end{align}
where
\begin{align}
\begin{split}
\tilde q_{m-1}^{(j)}(X) &= -(\tilde\nu_n^{(j,d)})^{-1}\,(-p)^{\delta_{3,j}+\delta_{4,j}} (x-N )^{-\delta_{1,j}-\delta_{3,j}}\pi(x) P_d^{(j)}(x),\\
\tilde q_{m}^{(j)}(X) &=  (\tilde\nu_n^{(j,d)})^{-1} (1-p)^{\delta_{3,j}+\delta_{4,j}}(x+1)^{\delta_{2,j}+\delta_{4,j}}\,\pi(x) P_d^{(j)}(x+1)
\end{split}
\end{align}
with 
\begin{equation}
\pi(x)=p(N-x)\dfrac{\eta^{(j)}(x)}{\xi^{(j)}(x) P_d^{(j)}(x)}(I-T^{-1})[q_{\pi}(x)].
\label{def:pi}
\end{equation}
Therefore, using \eqref{K:TRR} and \eqref{K:rec_variant}, one can exactly calculate the coefficients of the recurrence relation \eqref{XK:RR}.
\begin{corollary} 
Let $q_{\pi}(x)$ and $\pi (x)$ be polynomials satisfying \eqref{rec:spectrum} and \eqref{def:pi}, and $X$ be an operator which acts on a basis $\{ e_n\}_{n=0}^{\infty }$ as follows:
\begin{align}
 X \{e_n\}& = e_{n+1} + b_n e_{n} + u_n e_{n-1},\quad n=0,1,\ldots 
\end{align}
with $b_n = p(N-n)+n(1-p), u_n = (N+1-n)n  p (1-p)$.
We introduce the constants $\{\hat{c}_{n,\ell }^{(j,d)}\}$ defined by the following relation:
\begin{align}
\begin{split}
& \tilde q_{m-1}^{(j)}(X) \{e_{n+1} + (2n-N)(1-p) e_n + n(n-N-1)(1-p^2)e_{n-1}\} \\ 
&\qquad + \left(\tilde q_{m}^{(j)}(X)  + \tilde \nu_n^{(j,d)}  q_{\pi}(X) \right) \{e_n\} = \sum_{\ell=n-m}^{n+m} \hat{c}_{n,\ell}^{(j,d)}\, \tilde \nu_{\ell}^{(j,d)}\, e_{\ell}
\end{split}
\end{align}
For $(j,n) \ne (2,N+d-1), (4,-d-1)$, the following relation holds between $\{\hat{c}_{n,\ell }^{(j,d)}\}$ and  $\{c_{n,\ell }^{(j,d)}\}$ in \eqref{XK:RR}: 
\begin{equation}
\hat{c}_{n,\ell}^{(j,d)} = c_{n,\ell}^{(j,d)},\quad (j,\ell )\ne (3,N-d),(4,-d-1).
\end{equation}
\end{corollary}
It should be remarked here that 
$c_{n,N-d}^{(3,d)}$ and $c_{n,-d-1}^{(4,d)}$ can be automatically identified from the recurrence relation \eqref{qXK_n in K}.
Furthermore, $\{ c_{N+d-1,l}^{(2,d)}\}$ can be calculated in a similar manner by using \eqref{x-kraw2:N+d+1} and $\{ c_{-d-1,l}^{(4,d)}\}$ are given from \eqref{rec:spectrum} since $q_{\pi }(x)\hat{K}_{-d-1}^{(4,d)}(x)=q_{\pi }(x)$. 

\subsection{Concrete example: $K_{n}^{(2,2)}(x)$}
Here we give the explicit form of the properties of the $X$-Krawtchouk polynomials $\{ \hat{K}_{n}^{(j,d)}(x)\}$ with $j=2$ and $d=2$, that were used in \cite{mtv}.
For $x=-1,0,\ldots ,N$ and $n \in X_{2,2}=\{ 0,1,\ldots ,N,N+3\}$, the type 2 exceptional Krawtchouk polynomials $\{\hat{K}_n^{(2,2)}(x)\}_{n\in X_{2,2}}$ are defined by
\begin{equation}
\hat{K}_{n}^{(2,2)}(x)=\dfrac{1}{N+3-n}\left|\begin{array}{ll}
   (N-x)K_2(x-N+1;p,-N-2) & K_n(x) \\
   -(1+x)K_2(x-N+2;p,-N-2) & K_n(x+1)\\
   \end{array}\right|
\end{equation}
for $n=0,1,\ldots ,N$ and 
\begin{align}
\begin{split}
\hat{K}_{N+3}^{(2,2)}(x)=&\sum_{0 \le k,j \le 2} \frac{(-2)_j(-2)_k(p-1)^{2-k}p^{2-j}}{j!k!}
(-N-3)_{2-j}(-N-3)_{2-k} \\
&\times K_{N+k+j+1}(x+k+1;p,N+k+j+1).
\end{split}
\end{align}
The orthogonality relation is given by
\begin{equation}
 \sum_{x=-1}^{N}  \hat w^{(2,2)}(x) \hat{K}_n^{(2,2)}(x) \hat{K}_m^{(2,2)}(x) = \hat h_n^{(2,2)} \delta_{n,m},
\end{equation}
where 
\begin{align}
\begin{split}
  \hat w^{(2,2)}(x) 
&= \dfrac{w(x+1;p,N+1) }{K_2(x-N-1;p,-N-2)K_2(x-N;p,-N-2)} , \\
  \hat h^{(2,2)}_n &= \,2!\, (N+1)! \,(N+3)!\, (1-p)^{N+3} p^{N+3}.
  \end{split}
\end{align}
It can be easily shown that $K_2(x;p,-N-2)>0$ for $x \in (-N-2,0)$ and $p \in (0,1)$, which guarantees the positivity of the weight function $\hat w^{(2,2)}(x)$ on $x\in \{ -1,0,\ldots ,N\}$.
$\{\hat{K}_n^{(2,2)}(x)\}$ satisfies the following 7-term recurrence relation:
\begin{align}\label{ekraw:rec2}
\begin{split}
 q_3(x) \hat{K}_n^{(2,2)}(x)&= c_{n,3}\hat{K}_{n+3}^{(2,2)}(x)+c_{n,2} \hat{K}_{n+2}^{(2,2)}(x)+c_{n,1}\hat{K}_{n+1}^{(2,2)}(x)\\
 &+c_{n,0} \hat{K}_n^{(2,2)}(x)\\
 &+c_{n,-1} \hat{K}_{n-1}^{(2,2)}(x)+c_{n,-2}\hat{K}_{n-2}^{(2,2)}(x)+c_{n,-3} \hat{K}_{n-3}^{(2,2)}(x)
 \end{split}
 \end{align}
with
\begin{align}\label{ekraw:rec2-par}
\begin{split}
 q_3(x)&=K_3(x-N;p-N-1)-K_3(-1-N;p,-N-1),\\
 c_{n,3}&=1,\\
 c_{n,2}&=3(N-n+1)(2p-1),\\
 c_{n,1}&=3(N-n+2)
 \left\{ N-n+1-(4N-5n+2)p(1-p)\right\},\\
c_{n,0}&=-\sum_{ -3 \le \ell \le 3, \ell \ne 0 }\dfrac{(3-N)_{\ell}(-N)_{\ell} }{(1-N)_{\ell}}p^{\ell}c_{n,\ell},\\
 c_{n,-1}&=3 (N-n+1) (-n) (p-1)p(N-n+4)
 \left\{ N-n+2-(4N-5n+7)p(1-p)\right\},\\
 c_{n,-2}&=3(N-n+1)_2(-n)_2(p-1)^2p^2(N-n+5)(2p-1),\\
 c_{n,-3}&=(N-n+1)_2(-n)_3(p-1)^3p^3(N-n+6).  
\end{split}
\end{align}

Note that $\hat K_{N+1}^{(2,2)}(x)$ and $\hat K_{N+2}^{(2,2)}(x)$, which are zero-valued at $x=-1,0, \ldots, N$ on the grid, also appear in \eqref{ekraw:rec2}.  
Furthermore, when $\frac{1}{2}<p<1$, the coefficients $\{ c_{n,k}\}_{k\ne 0}$ are all positive and when $p=\frac{1}{2}$, $c_{n,2}=c_{n,-2}=0$ even though the weight function $\hat{w}^{(2,2)}(x)$ still takes positive value at $x=-1,0,\ldots ,N$. 

\section{Concluding Remarks}\label{sec6}
The exceptional Krawtchouk polynomials derived by a single-step Darboux transformation fall into four classes. Their construction confirmed that, like the other exceptional orthogonal polynomials, the resulting sequence of polynomials has a gap in degree. The weight function which determines their orthogonality is obtained by multiplying the weight function of the Krawtchouk polynomial by an appropriate rational function. 
Because of the symmetries of the Krawtchouk polynomials, there are various relations among the four classes of exceptional Krawtchouk polynomials.
Their factorization or Diophantine properties were also found and revealed the duality between case 1 and case 2 of the exceptional Krawtchouk polynomials, as well as between cases 3 and 4.
It was furthermore shown that the space spanned by the exceptional Krawtchouk polynomials can be characterized as a subspace of polynomials obtained from the action (on polynomials) of the exceptional Krawtchouk operator, thereby showing without using the limiting procedure from the exceptional Meixner polynomials, that there exist $2d+3$-term recurrence relations for the exceptional Krawtchouk polynomials.

A task that remains is the characterization of the multi-indexed Krawtchouk polynomials resulting from multi-step Darboux transformations.
Like the Krawtchouk polynomials that have many applications in probability theory, stochastic processes, coding theory, quantum mechanics, etc., the exceptional Krawtchouk polynomials presented in this paper are poised to be similarly useful.
This entails fascinating questions, some of which we plan to examine.

\section*{Acknowledgement}
The research of HM and ST is supported by JSPS KAKENHI (Grant Numbers 21H04073 and 19H01792 respectively) and that of LV by a discovery grant of the Natural Sciences and Engineering Research Council (NSERC) of Canada.

\bibliographystyle{tfs}
\bibliography{ekraw}

\end{document}